\documentclass[11pt,oneside]{amsart}
\usepackage{array}
\usepackage{multirow}
\usepackage{graphicx}
\usepackage{epsfig}
\usepackage{latexsym}
\usepackage{epsf}
\usepackage{graphicx,color,graphics}
\usepackage{amsmath,amssymb,amsthm,amsfonts}
\usepackage{stmaryrd}
\usepackage[latin1]{inputenc}
\usepackage{amsfonts}
\newtheorem{lemma}{Lemma}[section]

\newtheorem{theorem}[lemma]{Theorem}
\newtheorem{proposition}[lemma]{Proposition}

\newtheorem{remark}[lemma]{Remark}

\newcommand{\N}{\ifmmode{{\Bbb N}}\else{\mbox{${\Bbb N}$}}\fi}
\newcommand{\C}{\ifmmode{{\Bbb C}}\else{\mbox{${\Bbb C}$}}\fi}
\newcommand{\R}{\ifmmode{{\Bbb R}}\else{\mbox{${\Bbb R}$}}\fi}

\newcommand{\cal}{\mathcal}

\title{Exponential stability to the Bresse system with boundary dissipation
conditions}

\author{M. S. Alves} 
\address{Department of Mathematics.
Vi\c{c}osa University.  CEP 36570-000. Vi\c{c}osa. MG-Brazil.}
\email{{\tt malves@ufv.br}}

\author{Octavio Vera} 
\address{Department of Mathematics, Universidad del B{\'i}o-B{\'i}o, Av. Collao 1202, Casilla 5-C,
Concepci{\'o}n,Chile.} 
\email{{\tt overa@ubiobio.cl,\ octaviovera49@gmail.com} }

\author{Jaime Mu{\~n}oz Rivera}
\address{LNCC, Av. Getulio Vargas 333. Quintadinha. CEP 25651-075. Petr\'opilis.RJ.Brasil. Inst. de Matem\'atica, UFRJ, Av. da Silva Ramos. CEP 21945-970.RJ. Brasil}
\email{{\tt rivera@lncc.br, \ rivera@im.ufrj.br}}

\author{Amelie Rambaud}
\address{Department of Mathematics, grupo de investigaci\'on GIMNAP 151408/VC, Universidad del B{\'i}o-B{\'i}o, Av. Collao 1202, Casilla 5-C,
Concepci{\'o}n, Chile.} 
\email{{\tt arambaud@ubiobio.cl}}

\date{}

\begin{document}

\maketitle

\begin{abstract}
\noindent We consider the Bresse model with three control boundary
conditions. We prove the exponential stability of the system
 using the semigroup theory of linear operators and a
result obtained by Pr\"{u}ss \cite{z3TF84}.
\end{abstract}


\renewcommand{\theequation}{\thesection.\arabic{equation}}
\setcounter{equation}{0}
\section{Introduction}
In this work, we study  the stabilization of a problem arising from engineering motivation, the so-called  circular arch problem also known as the
Bresse system (see \cite{lagnese}) which is given by
\begin{eqnarray}
\label{eq1.1}
\begin{array}{ll}
\rho_{1}\,\varphi_{tt} - \kappa\,\left(\varphi_{x} + \psi +
\ell\,w\right)_{x} - k_{0}\,\ell\,(w_{x} - \ell\,\varphi)=0, &
\quad\mbox{in}\quad (0,\,L)\times (0,\,+\infty),
\\
\noalign{\medskip} \rho_{2}\,\psi_{tt} - b\,\psi_{xx} + \kappa\,
\left(\varphi_{x} + \psi + \ell\,w\right)=0 & \quad\mbox{in}\quad
(0,\,L)\times (0,\,+\infty),
\\
\noalign{\medskip}\rho_{1}\,w_{tt} - k_{0}\,(w_{x} -
\ell\,\varphi)_{x} + \kappa\,\ell\,\left(\varphi_{x} + \psi +
\ell\,w\right) = 0 & \quad\mbox{in}\quad (0,\,L)\times
(0,\,+\infty),
\end{array}
\end{eqnarray}
where $L$ is the length of the beam, $\rho_{1}=\rho\,A,$ $\rho_{2} =
\rho\,I$, $\kappa=\kappa'\,G\,A$, $\kappa_{0}=E\,A$, $b=E\,T$,
$\ell=R^{-1},$ $\rho$ is the density of the material, $E$ is the
modulus of elasticity, $G$ is the shear modulus, $\kappa'$ is the
shear factor, $A$ is the cross-sectional area, $I$  is the second
moment of area of the cross-section and $R$ is the radius of
curvature. The functions $w,\;\varphi$ and $\psi$ are the
longitudinal, vertical and shear angle displacements, respectively.

On of the main issues, both from a mathematical and physical point of view is the question of stability in long time ($t\rightarrow \infty$), in order to prevent the problem from infinite vibrations.  This question has been studied by many authors. We refer to the book of Liu and Zheng \cite{Liu} for a general survey on this topic. 

Concerning the Bresse system above, few results about the asymptotic behavior exist. Let us briefly review the different kinds of stabilization that have been conducted. An important problem in
the Bresse system is to find a minimum dissipation by which the
solution decays uniformly to zero in time. In this direction we have
the paper of Fatori and Rivera \cite{Lucy}, which improved the paper
by Liu and Rao \cite{Rao}, and more recently the article \cite{najda-whebe14}, where  the polynomial decay rate of the energy is improved. In these papers, the authors show that, in general, the Bresse
system is not exponentially stable but that there exists polynomial
stability with rates that depend on the wave propagations and the
regularity of the initial data. Moreover, they introduced a
necessary condition for the dissipative semigroup to decay polynomially.
This result allowed them to show some optimality to the polynomial
rate of decay. The Bresse system with frictional damping was
considered by  Alabau-Boussouira {\it et al.} \cite{dilberto1}. In
that paper the authors showed that the Bresse system is
exponentially stable if and only if the velocities of waves
propagations are the same. Also, they showed that when the
velocities are not the same, the system is not exponentially stable,
and they proved that the solution in this case goes to zero
polynomially, with rates that can be improved by taking more regular
initial data. This rate of polynomial decay was improved by Fatori
and Monteiro \cite{Monteiro}. The indefinite damping acting on the
shear angle displacement was considered by Palomino {\it et al.}
\cite{Juan}. In \cite{2479028} Noun and Wehbe extended the results
of Alabau-Boussouira {\it et al.} \cite{dilberto1} and considered
the important case when the dissipation law is locally distributed.
Finally, Lima {\it et al.} \cite{Mauro} considered the Bresse system
with past history acting in the shear angle displacement. They show
the exponential decay of the solution if and only if the wave speeds
are the same. If not, they show that the Bresse system is
polynomial stable with optimal decay rate.

In the present work, we attack  the delicate problem where the dissipative effect (some how the control we may have on the system)  takes place at the boundary. However, the consideration of only one dissipative effect, or only two, seems to be difficult to treat.  Let us mention some known results related to the boundary
stabilization of the Timoshenko beam. Kim and Renardy in
\cite{MR912448} proved the exponential stability of the  system
under two boundary controls. In \cite{t35M107}, Ammar-Khodja and his
co-authors studied the decay rate of the energy of the nonuniform
Timoshenko beam with two boundary controls acting in the
rotation-angle equation. In \cite{3303915}, Bassam and his
co-authors studied the indirect boundary stabilization of the
Timoshenko  system with only one dissipation law.

As a first step towards the stability of such systems with one control at the boundary, we consider in the present article the following boundary conditions that complement system \eqref{eq1.1}:
\begin{eqnarray}
\label{eq1.2}
\begin{array}{ll}
\varphi(L,\,t)=0, \quad \psi(L,\,t)=0, \quad w(L,\,t)=0 &
\quad\mbox{in}\quad(0,\,+\infty),
\\
\noalign{\medskip}
\kappa\,(\varphi_{x} + \psi + \ell\,w)(0,\,t) =
\gamma_{1}\,\varphi_{t}(0,\,t), & \quad\mbox{in}\quad(0,\,+\infty), \\
\noalign{\medskip}
b\,\psi_x(0,\,t) = \gamma_{2}\,\psi_{t}(0,\,t), &
\quad\mbox{in}\quad (0,\,+\infty, \\
\noalign{\medskip} k_{0}\,(w_{x} -
\ell\,\varphi)(0,\,t)=\gamma_{3}\,w_{t}(0,\,t), &
\quad\mbox{in}\quad (0,\,+\infty),
\end{array}
\end{eqnarray}
where $\gamma_{j}>0,\ j=1,\,2,\,3.$ In other words, we investigate three dissipative effects at the boundary.  The system is finally completed with  initial conditions
\begin{eqnarray}
\label{eq1.3}
\begin{array}{ll}
& \varphi(x,\,0)=\varphi_{0}(x), \quad
\varphi_{t}(x,\,0)=\varphi_{1}(x), \quad\mbox{in}\quad (0,\,L), \\
& \psi(x,\,0)=\psi_{0}(x),\quad\psi_{t}(x,\,0)=\psi_{1}(x),\quad
\mbox{in}\quad
(0,\,L),  \\
& w(x,\,0)=w_{0}(x), \quad w_{t}(x,\,0)=w_{1}(x),\quad
\mbox{in}\quad (0,\,L).
\end{array}
\end{eqnarray}
Let us define the energy functional associated to the system:  for $(\varphi,\,\psi,\,w)$ a regular solution to  \eqref{eq1.1}-\eqref{eq1.3}, its associated total energy is defined
by
\begin{eqnarray*}
\mathcal{E}(t)=\frac{1}{2}\int_{0}^{L}
\left(\rho_{1}\,|\varphi_{t}|^{2} + \rho_{2}\,|\psi_{t}|^{2} +
\rho_{1}\,|w_{t}|^{2} + \kappa\,|\varphi_{x} + \psi + \ell\,w|^{2} +
b\,|\psi_{x}|^{2} + k_{0}\,|w_{x} - \ell\,\varphi|^{2}\right)dx.
\end{eqnarray*}
Then a straightforward computation gives
\begin{eqnarray*}
\frac{d}{dt} \mathcal{E}(t)=-\ \gamma_{1} \,|\varphi_{t}(0)|^{2} -
\gamma_{2}\,|\psi_{t}(0)|^{2} - \gamma_{3}\,|w_{t}(0)|^{2} \leq 0,
\end{eqnarray*}
consequently the system \eqref{eq1.1}-\eqref{eq1.3} is dissipative
in the sense that the energy is non-increasing.
\begin{remark}
We observe that if $R \rightarrow \infty$  then $\ell\rightarrow 0$
and this model reduces to the well-know Timoshenko beam equations
(see \cite{graff} and \cite{lagnese} for details).
\end{remark}

 The main result of this paper is to prove that the exponential
stability of the system \eqref{eq1.1}--\eqref{eq1.3} holds. As far as
the authors know, there have been no contributions made in this
sense. Our main tools are semigroup techniques \cite{Pazy}, a result
by Pr\"{u}ss \cite{z3TF84} as well as spectral arguments.

The remaining part of this paper is organized as follows. Section 2
outlines briefly the notations and well-posedness of the system. In section 3, we show the exponential stability of the
corresponding semigroup. Through this paper, $C$ is a generic
constant, not necessarily the same at each occasion (it will change
line to line), which depends in an increasing way on the indicated
quantities.

\renewcommand{\theequation}{\thesection.\arabic{equation}}
\setcounter{equation}{0}
\section{Existence and uniqueness}\label{sec-C_0_semigroup}
The aim of this section is to prove the existence and uniqueness of
solutions for the problem \eqref{eq1.1}-\eqref{eq1.3}.

Given a Banach space $X$, let $\|\cdot\|_X$ be the usual norm
defined on $X$. In particular, we denote by $\langle \cdot,\,\cdot
\rangle$ and $\|\cdot\|$ the inner product and the norm defined on
$L^{2}(0,\,L)$, respectively. Before stating the existence and the
uniqueness result of problem \eqref{eq1.1}-\eqref{eq1.3}, we first
set-up the following short-hand notation for the function space
\begin{eqnarray*}
H_{L}^{1}(0,\,L)=\left\{\phi \in H^{1}(0,\,L):\quad
\phi(L)=0\right\}.
\end{eqnarray*}
Putting $\Phi=\varphi_{t}$ and
$\Psi=\psi_{t}$, the phase space of our problem is
\begin{eqnarray*}
\mathcal{H}=[H_{L}^{1}(0,\,L)]^{3}\times [L^{2}(0,\,L)]^{3},
\end{eqnarray*}
normed by
\begin{eqnarray*}
\|(\varphi,\,\psi,\,w,\,\Phi,\,\Psi,\,W)\|_{\mathcal{H}}^{2} & = &
\kappa\,\|\varphi_{x} + \psi + \ell\,w\|^{2} +
\rho_{1}\,\|\Phi\|^{2} + b\,\|\psi_{x}\|^{2} +
\rho_{2}\,\|\Psi\|^{2} + \rho_{1}\,\|W\|^{2} \\
&  & +\ k_{0}\,\|w_{x} - \ell\,\varphi\|^{2}.
\end{eqnarray*}
We denote by $C^{T}$ the transpose of a matrix $C$ and introducing
the state vector
\begin{eqnarray*}
U(t)=\left(\varphi(t),\,\psi(t),\,w(t),\,\Phi(t),\,\Psi(t),\,W(t)\right)^{T},
\end{eqnarray*}
system \eqref{eq1.1}-\eqref{eq1.2} can be written as a linear ordinary
differential equation in $\mathcal{H}$ of the form
\begin{equation}
\label{ODE} \frac{d}{dt}U(t)=\mathcal{A}\,U(t),
\end{equation}
where the domain $\mathcal{D}(\mathcal{A})$ of the linear operator
$\mathcal{A}:\mathcal{D}(\mathcal{A})\subseteq\mathcal{H} \to
\mathcal{H}$ is given by
\begin{eqnarray*}
\mathcal{D}(\mathcal{A}) & = & \left\{ U\in
\mathcal{H}:\quad\varphi,\,\psi,\,w\in H^{2}(0,\,L),\quad
\Phi,\,\Psi,\,W\in H_{L}^{1}(0,\,L),
\right.\\
&  & \quad\left. \kappa\,(\varphi_{x} + \psi +
\ell\,w)(0)=\gamma_{1}\,\Phi(0),\quad
b\,\psi_x(0)=\gamma_2\,\Phi(0),\; \right.
 \\
&  & \quad\left. \kappa_{0}\,(w_{x} -
\ell\,\varphi)(0)=\gamma_{3}\,W(0) \right\}
\end{eqnarray*}
and
\begin{eqnarray*}
\mathcal{A}U=\left (\begin{matrix} \Phi \\ \noalign{\medskip} \Psi
\\\noalign{\medskip}
W\\
\noalign{\medskip}
\displaystyle\frac{\kappa}{\rho_{1}}\left(\varphi_{x} + \psi +
\ell\,w\right)_{x} + \frac{k_{0}\,\ell}{\rho_{1}}\,(w_{x} - \ell\,w)
\\
\noalign{\medskip}
\displaystyle\frac{b}{\rho_{2}}\,\psi_{xx} -
\frac{\kappa}{\rho_{2}}\left(\varphi_{x} + \psi + \ell\,w\right)\\
\\
\noalign{\medskip} \displaystyle\frac{k_{0}}{\rho_{1}}\,(w_{x} -
\ell\,\varphi)_{x} - \frac{\kappa\,\ell}{\rho_{1}}\left(\varphi_{x}
+ \psi + \ell\,w\right)
\end{matrix}\right).
\end{eqnarray*}

\begin{proposition}\label{prop:1}
The operator $\mathcal{A}$ is the infinitesimal generator of a
contraction semigroup $\{ \mathcal{S}_{\mathcal{A}}(t)\}_{t\geq0}$.
\end{proposition}
\begin{proof}
The operator  $\mathcal{A}$ is dissipative. Indeed, for every $U \in
\mathcal{D}(\mathcal{A})$, it is not difficult to see that
\begin{equation}\label{diss}
\operatorname{Re}\langle \mathcal{A} U,\,U\rangle_{\mathcal{H}}= -\
\gamma_{1}\,|\Phi(0)|^{2} - \gamma_{2}\,|\Psi(0)|^{2} - \gamma_{3}\,
|W(0)|^{2}\leq 0.
\end{equation}
Moreover, the domain $\mathcal{D}$ of $\mathcal{A}$ is clearly dense in the Hilbert $\mathcal{H}$ and the operator is closed. Finally, for all
$F=(f_{1},\,f_{2},\,f_{3},\,f_{4},\,f_{5},\,f_{6})$ there exists a
unique $U=(\varphi,\,\psi,\,w,\,\Phi,\,\Psi,\,W)\in
\mathcal{D}(\mathcal{A}) $ such that $\mathcal{A}U=F$ (that is to say, that is solution to the resolvent system of the operator). Indeed, the system reads, in terms of components:
\begin{align}
& \Phi = f_{1},\label{uni1} \\
\noalign{\medskip}
& \Psi = f_{2}, \label{uni2}\\
\noalign{\medskip}
&  W = f_{3}, \label{uni3}\\
\noalign{\medskip} & \kappa \left(\varphi_{x} + \psi +
\ell\,w\right)_{x} + \kappa_{0}\, \ell\,(w_{x} -
\ell\,\varphi)=\rho_{1}\,f_{4},\label{uni4}
\\
\noalign{\medskip}
& b\,\psi_{xx} - \kappa\,\left(\varphi_{x} + \psi + \ell\,w\right) =
\rho_{2}\,f_{5}, \label{uni5} \\
\noalign{\medskip} & \kappa_{0}\,(w_{x} - \ell\,\varphi)_{x} -
\kappa\,\ell\left(\varphi_{x} + \psi + \ell\,w\right)=
\rho_{1}\,f_{6}. \label{uni6}
\end{align}
From \eqref{uni1}-\eqref{uni3} we know $\Phi,\;\Psi$ and $W$. To
show the existence and uniqueness of $(\varphi,\,\psi,\,w)$
satisfying \eqref{uni4}-\eqref{uni6} we  consider the continuous and
coercive sesquilinear form
\begin{eqnarray*}\mathbb{B}((\varphi,\,\psi,\,w),\,(u,\,v,\,p))
& = & \kappa\int_{0}^{L}(\varphi_{x} + \psi +
\ell\,w)\,(\overline{u_{x} + v + \ell\,p})\ dx +
b\int_{0}^{L}\psi_{x}\,\overline{v}_{x}\ dx
 \\
&  & +\ \kappa_{0}\int_{0}^{L}(w_{x} -
\ell\,\varphi)\,(\overline{p_{x} - \ell\,u})\ dx,
\end{eqnarray*}
for $(\varphi,\,\psi,\,w),\,(u,\,v,\,p)$ belong to $
[H_{L}^{1}(0,\,L)]^{3}$ and the continuous sesquilinear function
\begin{eqnarray*}
\mathbb{F}(u,\,v,\,p) & = & \rho_{1}\int_{0}^{L}f_{4}\,\overline{u}\
dx + \rho_{2}\int_{0}^{L}f_{5}\,\overline{v}\ dx +
\rho_{1}\int_{0}^{L}f_{6}\,\overline{p}\ dx + \gamma_{1}\,
f_{1}(0)\,\overline{u}(0)  \\
&  & +\ \gamma_{2}\,f_{2}(0)\,\overline{v}(0) + \gamma_{3}\,
f_{3}(0)\,\overline{p}(0).
\end{eqnarray*}
By the Lax-Milgram Theorem there exists a unique
$(\varphi,\,\psi,\,w)$  in  $ [H_{L}(0,\,L)^{1}]^{3}$ such that
\begin{eqnarray*}
\mathbb{B}((\varphi,\,\psi,\,w),\,(u,\,v,\,p))=\mathbb{F}(u,\,v,\,p),\quad\forall
\;(u,\,v,\,p)\in [H_{L}^{1}(0,\,L)]^{3}
\end{eqnarray*}
Hence  $ 0 \in \varrho( \mathcal{A})$, and the conclusion of Poposition \ref{prop:1} 
follows from the Lumer-Phillips Theorem (see for example \cite{Pazy}).

\end{proof}
As a direct consequence of Proposition \ref{prop:1}, we claim: 
\begin{theorem}
Given
$U_{0}=\left(\varphi_{0},\,\psi_{0},\,w_{0},\,\varphi_{1},\,\psi_{1},\,
w_{1}\right) \in \mathcal{H}$ there exists a unique solution $
U(t)={\cal
S}(t)U_{0}=\left(\varphi(t),\,\psi(t),\,w(t),\,\varphi_{t}(t),\,\psi_{t}(t),
\,w_t(t)\right)$ to \eqref{ODE} such that
\begin{equation*}
U\in C(0,\infty; \mathcal{H}).
\end{equation*}
If moreover, $\,\mathbf{U}_{0}\in  \mathcal{D} (\mathcal{A})$, then
\begin{equation*}
 U\in C^{1}([0,\,\infty[:\, \mathcal{H})\cap
C\big([0,\,\infty[:\,\mathcal{D} (\mathcal{A})\big). 
\end{equation*}
\end{theorem}

\renewcommand{\theequation}{\thesection.\arabic{equation}}
\setcounter{equation}{0}
\section{Exponential stability}\label{sec-expostab}

The main goal of this section is to prove the exponential decay of
solutions. Our main tool is the well known result
(see \cite{z3TF84}):
\begin{theorem}\label{pruss}
Let ${\cal S}(t)=e^{{\mathcal{A}}t}$ be a $C_{0}$-semigroup of
contractions on Hilbert space $\mathcal{H}.$ Then ${\cal S}(t)$ is
exponentially stable if and only if $i\,\R\subset \rho\,({\mathcal
A})$ and
\begin{equation}\label{estimativa}
\overline{\lim}_{\!\!\!\!\!\!\!\!\!\!\!\!\!{}_{{}_{{|\lambda
|\rightarrow \infty}}}}\|(i\,\lambda\,I - {\mathcal
A})^{-1}\|_{{\mathcal L}(\mathcal H)} < \infty.
\end{equation}
\end{theorem}

Therefore we will need to study the resolvent equation $
(i\,\lambda\,I - \mathcal{A}){U}={F}$, for $\lambda \ \in \ \mathbb{R}$, namely
\begin{align}
& i\,\lambda\,\varphi - \Phi = f_{1},\label{ris1} \\
 \noalign{\medskip}
& i\,\lambda\,\psi - \Psi = f_{2}, \label{ris2} \\
 \noalign{\medskip}
& i\,\lambda\,w - W = f_{3}, \label{ris3} \\
 \noalign{\medskip}
& i\,\lambda\,\rho_{1}\,\Phi - \kappa\, \left(\varphi_{x} + \psi +
\ell\,w\right)_{x} - \kappa_{0}\,\ell\,(w_{x} - \ell\,\varphi) =
\rho_{1}\,f_{4},\label{ris4}
\\
 \noalign{\medskip}
& i\,\lambda\,\rho_{2}\,\Psi - b\,\psi_{xx} +
\kappa\,\left(\varphi_{x} + \psi + \ell\,w\right)= \rho_{2}\,f_{5}, \label{ris5} \\
 \noalign{\medskip}
& i\,\lambda\,\rho_{1}\,W - \kappa_{0}\,(w_{x} - \ell\,\varphi)_{x}
+ \kappa\,\ell\,\left(\varphi_{x} + \psi + \ell\,w\right)=
\rho_{1}\,f_{6}, \label{ris6}
\end{align}
where $F=(f_{1},\,f_{2},\,f_{3},\,f_{4},\,f_{5},\,f_{6})^{T} \in
\mathcal{H}$. Taking inner product in $\mathcal{H}$ with $U$ and
using \eqref{diss} we get
\begin{equation}
\label{1estim} \left| \operatorname{Re}\langle \mathcal{A}
U,\,U\rangle_{\mathcal{H}}\right| \leq \|U\|_{\mathcal{H}}\,
\|F\|_{\mathcal{H}}.
\end{equation}
This implies that
\begin{equation}\label{Psi0}
|\Phi(0)|^{2} + |\Psi(0)|^{2} + |W(0)|^{2} \leq C\,
\|U\|_{\mathcal{H}}\,\|F\|_{\mathcal{H}},
\end{equation}
and, applying \eqref{ris1}-\eqref{ris3}, we obtain
\begin{equation}\label{psi_0}
|\varphi(0)|^{2} + |\psi(0)|^{2} + |w(0)|^{2} \leq
\frac{C}{|\lambda|^{2}}\,\|U\|_{\mathcal{H}}\,\|F\|_{\mathcal{H}} +
\frac{C}{|\lambda|^{2}}\,\|F\|_{\mathcal{H}}^{2}.
\end{equation}
Moreover, since
\begin{equation*}
|\varphi_{x}(0) + \psi(0) + \ell\,w(0)|^{2} + |\psi_{x}(0)|^{2} +
|w_{x}(0) - \ell\,\varphi(0)|^{2} \leq C\,\|U\|_{\mathcal{H}}\,
\|F\|_{\mathcal{H}},
\end{equation*}
it follows that
\begin{equation}\label{psix1}
|\varphi_{x}(0)|^{2} + |\psi_{x}(0)|^{2} + |w_{x}(0)|^{2} \leq C
\|U\|_{\mathcal{H}}\,\|F\|_{\mathcal{H}} + \frac{C}{|\lambda|^{2}}\,
\|U\|_{\mathcal{H}}\,\|F\|_{\mathcal{H}} + \frac{C}{|\lambda|^{2}}\,
\|F\|_{\mathcal{H}}^{2}.
\end{equation}
We will now establish a couple of lemmas in order to prove our stability result.

\begin{lemma}
\label{resolvente} The imaginary axis $i\,\R$ is contained in the
resolvent set $\rho\,({\mathcal A})$.
\end{lemma}
\begin{proof}
Because the domain of $\mathcal{A}$ has compact immersion over the
phase space $\mathcal{H}$, we only need to prove that there is no
imaginary eigenvalues. We will argue by contraction. Let us suppose
that there is $\lambda\in \R$, $\lambda \neq 0$, and $U\in {\cal
D}(\mathcal{A})$, $U\neq0$, such that $\mathcal{A}U=i\,\lambda\,U$.
Then, from \eqref{diss} we have
\begin{equation}\label{BC1}
 \Phi(0)=0,\quad\Psi(0)=0,\quad W(0)=0.
\end{equation}
Hence, from \eqref{ris1} and \eqref{eq1.3}$_2$ we obtain
\begin{equation}\label{BC2}
\varphi(0)=0,\quad\psi(0)=0,\quad w(0)=0 \quad \text{and} \quad
\varphi_{x}(0)=0, \quad\psi_{x}(0)=0,\quad w_{x}(0)=0.
\end{equation}
From \eqref{ris1}-\eqref{ris5} we have
\begin{align}\label{mag1}
\begin{array}{ll}
& -\ \lambda^{2}\,\rho_{1}\,\phi - \kappa\, \left(\varphi_{x} + \psi
+ \ell\,w\right)_{x} - \kappa_{0}\,\ell\,(w_{x} - \ell\,\varphi)=0,
\\
\noalign{\medskip}
& -\ \lambda^{2}\,\rho_{2}\,\psi - b\,\psi_{xx} +
\kappa\,\left(\varphi_{x} + \psi + \ell\,w\right)= 0, \\
\noalign{\medskip} & -\ \lambda^{2}\,\rho_{1}\,w -
\kappa_{0}\,(w_{x} - \ell\,\varphi)_{x} + \kappa\,\ell
\left(\varphi_{x} + \psi + \ell\,w\right)= 0.
\end{array}
\end{align}
Consider
$X=(\varphi,\,\psi,\,\omega,\,\varphi_{x},\,\psi_{x},\,\omega_{x})$.
Then we can rewrite \eqref{BC2} and \eqref{mag1} as the initial
value problem
\begin{equation}\label{edo}
\begin{array}{l}
\displaystyle\frac{d}{dx}X={\cal A}X,\\
 \noalign{\medskip}
X(0)=0,
\end{array}
\end{equation}
where
\begin{eqnarray*}
{\cal A}=\left(\begin{array}{cccccc}
0 & 0 & 0 & 1 & 0 & 0  \\
0 & 0 & 0 & 0 & 1 & 0  \\
0 & 0 & 0 & 0 & 0 & 1  \\
\frac{k_{0}\,\ell^{2}}{\kappa} & -\ 1 & 0 &
-\ \frac{\rho_{1}\,\lambda^{2}}{\kappa} & 0 &
-\ \frac{(k_{0} + \kappa)\,\ell}{\kappa}  \\
0 & \frac{-\ \rho_{2}\,\lambda^{2} + \kappa}{b} & \frac{\kappa l}{b}
& \frac{\kappa}{b} & 0 & 0  \\
0 & \frac{\kappa\,\ell}{k_{0}} & \frac{-\ \rho_{1}\,\lambda^{2} +
\kappa\,\ell^{2}}{k_{0}} & \frac{(k_{0} + \kappa)\,\ell}{k_{0}} & 0
& 0
\end{array}
\right).
\end{eqnarray*}
By the Picard Theorem for ordinary differential equations the system
\eqref{edo} has a unique solution $X=0$. Therefore
$\varphi=0,\;\psi=0,\;w=0$. It follows from
\eqref{ris1}-\eqref{ris3}, for $f_{1}=f_{2}=f_{3}=0$, that
$\Phi=0,\;\Psi=0,\;W=0$, i.e., $U=0$. \\
\end{proof}

Let us introduce the following notation
\begin{align*}
& \mathcal{I}_{\varphi}(\alpha) = \rho_{1}\,|\Phi(\alpha)|^{2} +
\kappa\,|\varphi_{x}(\alpha)|^{2},  \\
& \mathcal{I}_{\psi}(\alpha) = \rho_{2}\,|\Psi(\alpha)|^{2}
+ b\,|\psi_{x}(\alpha)|^{2},\\
&  \mathcal{I}_{w}(\alpha) = \rho_{1}\,|W(\alpha)|^{2}
+ \kappa_{0}\,|w_{x}(\alpha)|^{2},  \\
&  \mathcal{I}(\alpha) = \mathcal{I}_{\varphi}(\alpha)
+ \mathcal{I}_{\psi}(\alpha) + \mathcal{I}_{w}(\alpha)  \\
& \mathcal{E}_{\psi}(L)=\int_{0}^{L}\mathcal{I}_{\psi}(s)\
ds,\quad\mathcal{E}_{\varphi}(L) =
\int_{0}^{L}\mathcal{I}_{\varphi}(s)\ ds,\quad\mathcal{E}_{w}^{n}(L)
= \int_{0}^{L}\mathcal{I}_{w}(s)\ ds.
\end{align*}

\begin{lemma}\label{marg1}
Let  $q\in H^{1}(0,\,L)$. We have that
\begin{align}\label{bb1}
\mathcal{E}_{\varphi}(L) = &
\left.q\,\mathcal{I}_{\varphi}\,\right|_{0}^{L} -
\left.\kappa_{0}\,\ell^{2}\,q\, |\varphi|^{2}\,\right|_{0}^{L} +
2\,\kappa\,\mbox{Re}\int_{0}^{L}q\,\psi_{x}
\,\overline{\varphi}_{x}\ dx + \kappa_{0}\,\ell^{2}
\int_{0}^{L}q'(x)\,|\varphi|^{2}\ dx  \\
& +\ 2\,(\kappa + \kappa_{0}
)\,\ell\,\mbox{Re}\int_{0}^{L}q\,w_{x}\,\overline{\varphi}_{x}\ dx +
R_{1} \nonumber
\end{align}
\begin{eqnarray}
\label{bb2} \mathcal{E}_{\psi}(L) & =
&\left.q\,\mathcal{I}_{\psi}\,\right|_{0}^{L} - \left.\kappa\,
q\,|\psi|^{2}\,\right|_{0}^{L} - 2\,\kappa\,\mbox{Re}\int_{0}^{L}
q\,\varphi_{x}\,\overline{\psi}_{x}\ dx  \nonumber \\
&  & +\ \kappa\int_{0}^{L}q'(s)\,|\psi|^{2}\ dx - 2\,\kappa\,\ell\,
\mbox{Re}\int_{0}^{L}q\,w\,\overline{\psi}_{x}\ dx + R_{2}.
\end{eqnarray}
and
\begin{eqnarray}\label{bb3}
\mathcal{E}_{w}(L) & = &\left. q\,\mathcal{I}_{w}\,\right |_{0}^{L}
- \left.\kappa\,\ell^{2}\,q\,|w|^{2}\,\right|_{0}^{L} -
2\,\kappa\,\ell\,\mbox{Re}\int_{0}^{L}q\,\psi\,\overline{w}_{x}\ dx
\nonumber  \\
&  & -\ 2\,(\kappa +
k_{0})\,\ell\,\mbox{Re}\int_{0}^{L}q\,\varphi_{x}\,\overline{w}_{x}\
dx + \kappa\,\ell^{2}\int_{0}^{L}q'(s)\,|w|^{2}\ dx + R_{3},
\end{eqnarray}
where $R_{i}$ satisfies
\begin{eqnarray*}
|R_{i}|\leq C\,\|U\|\,\|F\|, \quad i=1,\;2,\;3,
\end{eqnarray*}
for a positive constant $C$.
\end{lemma}
\begin{proof}
To get \eqref{bb1}, let us multiply the equation \eqref{ris2} by
$q\,\overline{\varphi}_{x}.$  Integrating on $(0,\,L)$ we obtain
\begin{eqnarray*}
&  &
i\,\lambda\,\rho_{1}\int_{0}^{L}\Phi\,q\,\overline{\varphi}_{x}\ dx
- \kappa\int_{0}^{L}(\varphi_{x} + \psi +
\ell\,w)_{x}\,q\,\overline{\varphi}_{x}\ dx \\
&  & -\ \kappa_{0}\, \ell\int_{0}^{L}(w_{x} -
\ell\,\varphi)\,q\,\overline{\varphi}_{x}\ dx =
\rho_{1}\int_{0}^{L}f_{4}\,q\,\overline{\varphi}_{x}\ dx
\end{eqnarray*}
or
\begin{align*}
& -\
\rho_{1}\int_{0}^{L}\Phi\,q\,(\overline{i\,\lambda\,\varphi_{x}})\
dx - \kappa\int_{0}^{L}q\,\varphi_{xx}\,\overline{\varphi}_{x}\ dx -
\kappa\int_{0}^{L}q\,\psi_{x}\,\overline{\varphi}_{x}\ dx \\
& -\ (\kappa +
\kappa_{0})\,\ell\int_{0}^{L}q\,w_{x}\,\overline{\varphi}_{x}\ dx +
\kappa_{0}\,\ell^{2}\int_{0}^{L}q\,\varphi\,\overline{\varphi}_{x} =
\rho_{1}\int_{0}^{L}f_{4} \,q\,\overline{\varphi}_{x}\ dx.
\end{align*}
Since $i\,\lambda\,\varphi_{x} = \Phi_{x} + f_{1x} $  taking the
real part in the above equality results in
\begin{align*}
& -\ \frac{\rho_{1}}{2}\int_{0}^{L}q\,\frac{d}{dx}|\Phi|^{2}\ dx -
\frac{\kappa}{2}\int_{0}^{L}q\,\frac{d}{dx}|\varphi_{x}|^{2}\ dx
= \rho_{1}\,\mbox{Re}\int_{0}^{L}f_{4}\,q\,\overline{\varphi}_{x}\ dx  \\
& +\ \rho_{1}\,\mbox{Re}\int_{0}^{L}\Phi\,q\,\overline{f}_{1x}\ dx +
\kappa\,\mbox{Re}\int_{0}^{L}q\,\psi_{x}\,\overline{\varphi}_{x}\ dx
+ (\kappa + \kappa_{0})\,\ell\,\mbox{Re}\int_{0}^{L}q\,w_{x}\,
\overline{\varphi}_{x}\ dx  \\
& -\ \frac{\kappa_{0}\,
\ell^{2}}{2}\int_{0}^{L}q\,\frac{d}{dx}|\varphi|^{2}.
\end{align*}
Performing an integration by parts we get
\begin{eqnarray*}\label{Id1}
\lefteqn{\int_{0}^{L}q'(s)\,[\rho_{1}\,|\Phi(s)|^{2} +
\kappa\,|\varphi_{x}(s)|^{2}]\ ds }  \\
& = &
\left.q\,\mathcal{I}_{\varphi}\,\right|_{0}^{L} -
\left.\kappa_{0}\,\ell^{2}\,q\,|\varphi|^{2}\,\right|_{0}^{L}
+ 2\,\kappa\,\mbox{Re}\int_{0}^{L}q\,\psi_{x}\,\overline{\varphi}_{x}\ dx  \\
&  & +\ \kappa_{0}\,\ell^{2}\int_{0}^{L}q'(s)\,|\varphi|^{2} +
2\,(\kappa + \kappa_{0})\,\ell\,\mbox{Re}
\int_{0}^{L}q\,w_{x}\,\overline{\varphi}_{x}\ dx + R_{1}
\end{eqnarray*}
where
\begin{eqnarray*}
R_{1} = 2\,\rho_{1}\,\mbox{Re}\int_{0}^{L}\Phi\,q\,
\overline{f}_{1x}\ dx + 2\,\rho_{1}\,\mbox{Re}\int_{0}^{L}f_{4}
\,q\,\overline{\varphi}_{x}\ dx
\end{eqnarray*}
Similarly, multiplying equation \eqref{ris4} by
$q\,\overline{\psi}_{x}$, integrating on $(0,\,L)$ and taking the
real part we obtain
\begin{align*}
&-\ \frac{\rho_{2}}{2}\int_{0}^{L}q\,\frac{d}{dx}|\Psi|^{2}\ dx
- \frac{b}{2}\int_{0}^{L}q\,\frac{d}{dx}|\psi_{x}|^{2}\ dx
= \rho_{2}\,\mbox{Re}\int_{0}^{L}f_{5}\,q\,\overline{\psi}_{x}\ dx  \\
&\qquad \qquad+\ \rho_{2}\,\mbox{Re}\int_{0}^{L}\Psi\,q\,\overline{f}_{2x}\ dx
- \kappa\,\mbox{Re}\int_{0}^{L}q\,\psi_{x}\,\overline{\varphi}_{x}\ dx
- \kappa\,\ell\mbox{Re}\int_{0}^{L}q\,w\,\overline{\psi}_{x}\ dx  \\
&\qquad \qquad -\
\frac{\kappa}{2}\int_{0}^{L}q\,\frac{d}{dx}|\psi|^{2}.
\end{align*}
Performing an integration by parts we obtain
\begin{eqnarray*}\label{Id2}
\lefteqn{\int_{0}^{L}q'(s)[\rho_{2}\,|\Psi(s)|^{2} +
b\,|\psi_{x}(s)|^{2}]\ ds }  \\
& = &\left.q\,\mathcal{I}_{\psi}\,\right|_{0}^{L} -
\left.\kappa\,q\,|\psi|^{2}\,\right|_{0}^{L} -
2\,\kappa\,\mbox{Re}\int_{0}^{L}q\,\varphi_{x}\,\overline{\psi}_{x}\
dx \\
&  &+\ \kappa\int_{0}^{L}q'(s)\,|\psi|^{2}\ dx - 2\,\kappa\,\ell\,
\mbox{Re}\int_{0}^{L}q\,w\,\overline{\psi}_{x}\ dx + R_{2}
\end{eqnarray*}
where
\begin{eqnarray*}
R_{2} = 2\,\rho_{2}\,\mbox{Re}\int_{0}^{L}\Psi\,q\,
\overline{f}_{2x}\ dx + 2\,\rho_{2}\,\mbox{Re}\int_{0}^{L}f_{5}
\,q\,\overline{\psi}_{x}\ dx.
\end{eqnarray*}
Finally,  multiplying equation \eqref{ris5} by
$q\,\overline{w}_{x}$, integrating on $(0,\,L)$ and taking the real
part, after some algebric manipulations we obtain \eqref{bb3} for
\begin{eqnarray*}
R_{3} = 2\,\rho_{1}\,\mbox{Re}\int_{0}^{L}W\,q\, \overline{f}_{3x}\
dx + 2\,\rho_{1}\,\mbox{Re}\int_{0}^{L}f_{6} \,q\,\overline{w}_{x}\
dx.
\end{eqnarray*}
Our conclusion follows.
\end{proof}

We are now ready to state our main stability result. 
\begin{theorem}
The semigroup $\{{\cal S}_{\mathcal{A}}(t)\}_{t\geq0}$ is
exponentially stable, that is, there exist positive constants $M$
and $\mu$ such that
\begin{eqnarray*}
\|{\cal S}_{\mathcal{A}}(t)\|_{\mathcal{L}(\mathcal{H})}\leq M \exp
(-\ \mu\,t),\ \forall\,t\geq 0.
\end{eqnarray*}
\end{theorem}
\begin{proof} By Lemma \ref{resolvente} we know that
$i\,\mathbb{R}\subset \rho(\mathcal{A})$. Therefore, by Theorem
\ref{pruss} it sufficient to show that the estimate
\eqref{estimativa} holds. Given
$F=(f_{1},\,f_{2},\,f_{3},\,f_{4},\,f_{5},\,f_{6})\in \mathcal{H}$
and $\lambda \in \mathbb{R}$ let be $U$ the unique function
satisfying
\begin{eqnarray*}
(i\,\lambda\,I - \mathcal{A})U=F.
\end{eqnarray*}
If we take $q(x)= x - \ell$ in Lemma \ref{marg1} and if we add
\eqref{bb1}-\eqref{bb3} we arrive at
\begin{eqnarray*}
\lefteqn{\mathcal{E}_{\varphi}(L) + \mathcal{E}_{\psi}(L) +
\mathcal{E}_{w}(L) }  \\
& = &\,
L\,\mathcal{I}_{\varphi}(0) - \kappa_{0}\,\ell^{2}\,L\,|\varphi(0)|^{2}
+ \kappa_{0}\,\ell^{2}\int_{0}^{L}|\varphi|^{2}\ dx  \\
&  & +\ L \,\mathcal{I}_{\psi}(0) - L\,\kappa\,|\psi(0)|^{2} +
\kappa\int_{0}^{L}|\psi|^{2}\ dx + L\,\mathcal{I}_{w}(0) -
\kappa\,\ell^{2}\,L\,|w(0)|^{2}  \nonumber  \\
&  & +\ \kappa\,\ell\int_{0}^{L}|w|^{2}\ dx + R_{1} + R_{2} + R_{3} \\
&  & -\ 2\,\kappa\,\ell\,\mbox{Re}\int_{0}^{L}(x - L)\,
w\,\overline{\psi}_{x}\ dx - 2\,\kappa\,\ell\,\mbox{Re}\int_{0}^{L}
(x - L)\,\psi\,\overline{w}_{x}\ dx.
\end{eqnarray*}
Since
\begin{eqnarray*}
&  & -\
2\,\kappa\,\ell\,\mbox{Re}\int_{0}^{L}q\,w\,\overline{\psi}_{x}\ dx
- 2\,\kappa\,\ell\,\mbox{Re}\int_{0}^{L}q\, \psi\,\overline{w}_{x}\
dx \\
& = &  -\ 2\,\kappa\,\ell\,L\,\mbox{Re}\,w(0) \overline{\psi}(0) +
2\,\kappa\,\ell\,\mbox{Re}\int_{0}^{L}\psi\, \overline{w}\ dx
\end{eqnarray*}
using Lemma \ref{marg1} and the Young inequality we get
\begin{align*}
\lefteqn{\mathcal{E}_{\varphi}(L) + \mathcal{E}_{\psi}(L) +
\mathcal{E}_{w}(L) }\\
\leq &\,
L\,\mathcal{I}_{\varphi}(0) + \kappa_{0}\,\ell^{2}\int_{0}^{L}|\varphi|^{2}\ dx  \\
& +\ L \,\mathcal{I}_{\psi}(0) + \kappa\,\ell\,|\psi(0)|^{2} +
\kappa\,(1 + \ell)\int_{0}^{L}|\psi|^{2}\ dx
+ L\,\mathcal{I}_{w}(0) + \kappa\,\ell\,|w(0)|^{2}  \nonumber  \\
& +\ 2\,\kappa\,\ell\int_{0}^{L}|w|^{2}\ dx +
C\,\|U\|_{\mathcal{H}}\,\|F\|_{\mathcal{H}}   \\
\end{align*}
for a positive constant $C$. It results by \eqref{Psi0},
\eqref{psi_0} and \eqref{psix1} that we can find a positive constant
$C$ such that
\begin{align*}
\lefteqn{\mathcal{E}_{\varphi}(L) + \mathcal{E}_{\psi}(L) +
\mathcal{E}_{w}(L) } \\
\leq &\,
\kappa_{0}\,\ell^{2}\int_{0}^{L}|\varphi|^{2}\ dx
+ \kappa\,(1 + \ell)\int_{0}^{L}|\psi|^{2}\ dx
+ 2\,\kappa\,\ell\int_{0}^{L}|w|^{2}\ dx  \\
& +\ \frac{C}{|\lambda|^{2}}\,\|U\|_{\mathcal{H}}\,\|
F\|_{\mathcal{H}} + C\,\|U\|_{\mathcal{H}}\,\| F\|_{\mathcal{H}} +
\frac{C}{|\lambda|^{2}}\,\|F\|_{\mathcal{H}}^{2},
\end{align*}
for $\lambda \neq 0.$ Since that $\varphi = \frac{\Phi +
f_{1}}{i\,\lambda}$, $\psi=\frac{\Psi + f_{2}}{i\,\lambda}$ and $w =
\frac{W + f_{3}}{i\,\lambda}$ we obtain
\begin{align*}
\|U\|_{\mathcal{H}}^{2}\leq & \frac{C}{|\lambda|^{2}}\,\|
U\|_{\mathcal{H}}^{2} + \frac{C}{|\lambda|^{2}}\,\|
F\|_{\mathcal{H}}^{2} + \frac{C}{|\lambda|^{2}}\,\|
U\|_{\mathcal{H}}\,\|F\|_{\mathcal{H}} + C\,\|F\|_{\mathcal{H}}^{2}
\end{align*}
for $\lambda \neq 0.$ If $|\lambda|>1$ we get
\begin{align*}
\left(1 - \frac{C}{|\lambda|}\right)\|U\|_{\mathcal{H}}^{2} \leq C\,
\|F\|_{\mathcal{H}}^{2}.
\end{align*}
Consequently, since $\lambda \, \mapsto \,  (i\,\lambda\,I -
\mathcal{A})$ is continuous it follows that
\begin{align*}
\|(i\,\lambda\,I - \mathcal{A})\|_{\mathcal{L(H)}}\leq C, \ \forall
\,\lambda \in \mathbb{R},
\end{align*}
for a positive constant $C$. The conclusion then follows by applying the Theorem \ref{pruss}.
\end{proof}

\section{Conclusion}
In this paper, we provide a result of exponential stability for the Bresse system when three dissipative effects are concentrated at the boundary. It is a step towards complete understanding of  boundary stabilization of such system. Indeed, we expect to be able to obtain similar results as the ones existing for Timoshenko type models  \cite{t35M107, 3303915,MR912448}, but it seems for now, that there are more mathematical difficulties for the Bresse model.   

\section*{Acknowledgments}
This work of M. S. Alves has been supported by  CNPq-Brazil:
Processo  $158706/2014-5.$  The authors thank the National
Laboratory for Scientific Computation (LNCC/MCT),
Petr\'{o}polis-Brazil, for its hospitality during the stage of
visiting professors. Octavio Vera thanks the support of the Fondecyt
project 1121120. Amelie Rambaud thanks the support of the Fondecyt project 11130378.

\end{document}